\RequirePackage[l2tabu, orthodox]{nag}
\documentclass[11pt]{amsart} %{{{1

\author{Andy Hammerlindl}
\address{School of Mathematical Sciences, Monash University, Victoria 3800 Australia} \urladdr{ http://users.monash.edu.au/~ahammerl/}  \email{andy.hammerlindl@monash.edu}
\title{Bundle switching in higher dimensions}
\date{\today}
\usepackage{amsmath}
\usepackage{amssymb}
\usepackage{amsfonts}
\usepackage{amsthm}
\usepackage{amscd}
\usepackage{graphicx}
\usepackage{setspace}
\usepackage{enumitem}
\usepackage{cleveref}
\usepackage{multicol}
\usepackage{fourier}
\crefname{figure}{figure}{figure}
\Crefname{figure}{Figure}{Figure}

\usepackage[T1]{fontenc}

\makeatletter
\def\saveenum{\xdef\@savedenum{\the\c@enumi\relax}}
\def\resetenum{\global\c@enumi\@savedenum}
\makeatother

    \newcommand{\bbR}{\mathbb{R}}

    \newcommand{\bbZ}{\mathbb{Z}}

    \newcommand{\bbN}{\mathbb{N}}
    \newcommand{\bbT}{\mathbb{T}}

    \newcommand{\subof}{\subset}
    \newcommand{\ti}{\times}
    \newcommand{\sans}{\setminus}

    \newcommand{\Es}{E^s}
    \newcommand{\Ec}{E^c}
    \newcommand{\Eu}{E^u}
    \newcommand{\Ecu}{E^{cu}}
    \newcommand{\Ecs}{E^{cs}}

    \newcommand{\Esf}{E^s_f}
    \newcommand{\Ecf}{E^c_f}
    \newcommand{\Euf}{E^u_f}
    \newcommand{\Ecsf}{E^{cs}_f}
    \newcommand{\Ecuf}{E^{cu}_f}

    \newcommand{\Esg}{E^s_g}
    \newcommand{\Ecg}{E^c_g}
    \newcommand{\Eug}{E^u_g}
    \newcommand{\Ecsg}{E^{c\mykern s}_g}
    \newcommand{\Ecug}{E^{c\mykern u}_g}

    \newcommand{\Essg}{E^{s\mykern s}_g}
    \newcommand{\Ewsg}{E^{w\! s}_g}
    \newcommand{\Eccg}{E^{c\mykern c}_g}

    \newcommand{\vsss}{v_{s\mykern s}}
    \newcommand{\vws}{v_{w\mykern s}}
    \newcommand{\vcu}{v_{c\mykern u}}

    \newcommand{\Esphi}{E^s_\phi}
    \newcommand{\Ecphi}{E^c_\phi}
    \newcommand{\Euphi}{E^u_\phi}

    \newcommand{\Wsg}{W^s_g}

    \newcommand{\cF}{\mathcal{F}}

    \newcommand{\Aone}{\mathcal{A}}
    \newcommand{\Bone}{\mathcal{B}}
    \newcommand{\Cone}{\mathcal{C}}
    \newcommand{\Eone}{\mathcal{E}}
    \newcommand{\Uone}{\mathcal{U}}
    \newcommand{\cc}{\subset \subset}

    \newcommand{\ep}{\epsilon}
    \newcommand{\lam}{\lambda}

    \newcommand{\gam}{\gamma}
    
    \newcommand{\sig}{\sigma}

    \newcommand{\al}{\alpha}
    
    \newcommand{\bt}{\beta}

    \newcommand{\qandq}{\quad \text{and} \quad}

    \newcommand{\length}{\operatorname{length}}

    \newcommand{\mykern}{\mkern-0.5mu}

\numberwithin{equation}{section}
\newtheorem{thm}[equation]{Theorem}
\newtheorem{cor}[equation]{Corollary}
\newtheorem{lemma}[equation]{Lemma}

\newtheorem{prop}[equation]{Proposition}
\newtheorem{question}[equation]{\textbf{Question}}

\newtheorem{addendum}[equation]{\textbf{Addendum}}

\theoremstyle{remark}
\newtheorem*{remark} {\textbf{Remark}}

\newtheorem*{notation} {\textbf{Notation}}

\providecommand{\acknowledgement}{{\noindent \textbf{Acknowledgements.}}\quad}

\begin{document}

%{\small
%{\sc disclaimer.}
%\ \ This paper is still in preparation.
%Please, contact me if you have any questions.
%Comments and suggestions for improvement are also most welcome.
%}
%
%\bigskip

% Abstract {{{1

\begin{abstract}
    Assuming it preserves an orientation of its stable bundle,
    any three-dimensional partially hyperbolic diffeomorphism can be used to
    construct a four-dimensional partially hyperbolic diffeomorphism which is
    dynamically incoherent.
    Under the same assumption,
    the time-one map of any three-dimensional Anosov flow can be used to
    construct a four-dimensional diffeomorphism which is both absolutely
    partially hyperbolic and dynamically incoherent.
    Further results hold in higher dimensions 
    under additional assumptions.
\end{abstract}
\maketitle

\section{Introduction} %{{{1

Partially hyperbolic systems are an intensively studied class of
cha\-otic dyna\-mical systems, with close links to robust transitivity and stable
ergodicity \cite{buhbook}. In recent years, significant progress has been made
understanding and classifying partially hyperbolic systems in dimension 3
\cite{hp2018survey, crhrhu2018survey, 1908.06227, 2008.04871}.
However, partial hyperbolicity in higher dimensions is far less understood.
This paper looks at the question of dynamical coherence for partially
hyperbolic systems in dimensions 4 and higher.

A diffeomorphism $f : M \to M$ on a closed Riemannian manifold $M$ is
\emph{partially hyperbolic} if
there are $k \ge 1$ and an invariant splitting
of the tangent bundle $TM$ into three subbundles
$TM = \Es \oplus \Ec \oplus \Eu$ such that
\[  \| Df^k v^s \| < 1 < \| Df^k v^u \|
    \qandq
    \| Df^k v^s \| < \| Df^k v^c \| < \| Df^k v^u \|
\]
for all $p \in M$ and unit vectors
$v^s \in \Es_p$, $v^c \in \Ec_p$, and $v^u \in \Eu_p$.
Up to replacing the metric on $M,$ we can freely assume that $k = 1.$

There are always invariant foliations tangent to the stable direction $\Es$
and the unstable direction $\Eu,$
but there may or may not be a foliation tangent to the center direction.
A partially hyperbolic diffeomorphism is \emph{dynamically coherent}
if there are invariant foliations tangent to the
center-stable $\Ecs = \Ec \oplus \Es$ and
center-unstable $\Ecu = \Ec \oplus \Eu$ bundles.
The intersection of these two foliations then yields
an invariant center direction.

Determining whether or not a system is dynamically coherent
is often a key first step in understanding and classifying
its dynamical behaviour.
For decades, it was an open question if a partially hyperbolic system
with a one-dimensional center was necessarily dynamically coherent.
This was finally answered in the negative by
F.~Rod\-riguez-Hertz, J.~Rodriguez-Hertz, and R.~Ures \cite{RHRHU-nondyn}.
They constructed a partially hyperbolic system on the 3-torus $\bbT^3$
having an invariant embedded 2-torus $\bbT^2$ tangent to $\Ecu$
and further showed that there was no foliation tangent to $\Ecu.$
This is an example of a \emph{cu-submanifold},
a compact embedded submanifold tangent to $\Ecu$.
A \emph{cs-submanifold} is defined analogously.

For partially hyperbolic systems in dimension 3,
having a 2-dimensional \emph{cs} or \emph{cu}-submanifold places
severe restrictions on the manifold $M$ supporting $f$ \cite{RHRHU-tori}.
Moreover, the dynamics of these systems have been completely classified
\cite{hp2019class}.

In this paper, we investigate the construction of dynamically incoherent
examples in dimensions 4 and higher,
showing that the situation here is very different from dimension 3.
In particular, any partially hyperbolic diffeomorphism $g$ in dimension 3
has a one-dimensional stable bundle.
If $g$ preserves the orientation of its stable bundle,
then it can be used to construct
a dynamically incoherent example in dimension 4.
More generally, we have the following.

\begin{thm} \label{thm:oneswitch}
    Suppose that $g : M \to M$ is a partially hyperbolic diffeomorphism such that
    the stable bundle $\Esg$ is one-dimensional
    and $g$ preserves the orientation of $\Esg.$
    Then there is a partially hyperbolic diffeomorphism
    $f : M \ti S^1 \to M \ti S^1$ such that
    \begin{enumerate}
        \item $M \ti \{0\}$ is a cu-submanifold,
        \item
        $f(x,0) = (g(x), 0)$ for all $x \in M,$ and
        \item
        $f$ is dynamically incoherent.
    \end{enumerate} \end{thm}

We refer to the technique used both in \cite{RHRHU-nondyn} and
here in \cref{thm:oneswitch} as ``bundle switching''.
The construction, roughly speaking,
starts with a direct product $g \ti h : M \ti S^1 \to M \ti S^1$
where $h : S^1 \to S^1$ has a strong contraction at a fixed point $0 \in S^1$
and weak expansion at another fixed point $1 \in S^1.$
(Throughout this paper, we regard $S^1$ as a quotient of the interval $[-1, 1]$
identifying the endpoints $-1$ and $1$.)

On $M \ti \{0\},$ the strong contraction of $h$
means that the stable bundle is in the direction of $S^1$ fibers
and the center direction lies in the tangent space of $M \ti \{0\}.$

On $M \ti \{1\},$ the weak expansion of $h$
means that the $S^1$ fibers will be tangent to the center direction
and the stable direction lies in the tangent space of $M \ti \{1\}.$

To produce a global partially hyperbolic splitting,
the actual diffeomorphism $f : M \ti S^1 \to M \ti S^1$
applies a shear in the regions between $M \ti \{0\}$ and $M \ti \{1\}$
in order to ``switch'' the alignments of the center and stable bundles.

In the original construction in \cite{RHRHU-nondyn},
the manifold $M$ is the 2-torus $\bbT^2$ and
$g : \bbT^2 \to \bbT^2$ is given by a linear toral automorphism.
For this map, the stable bundle $\Esg$ is smooth (linear, in fact)
and the construction shears the dynamics exactly in the direction of $\Esg.$

For a general partially hyperbolic splitting,
$TM = \Esg \oplus \Ecg \oplus \Eug,$
the bundles are only H\"older continuous
and not $C^1$ \cite{hw1999prevalence}.
Therefore the construction in \cref{thm:oneswitch}
uses a shear along a smooth vector field approximating
the stable direction.
This means that proving that the resulting diffeomorphism $f$ is
partially hyperbolic is more delicate,
but the proof can still be achieved.

In the setting of \cref{thm:oneswitch}, if $h : M \to M$ is a diffeomorphism
which commutes with $g,$ then we can cut $M \ti S^1$
open along $M \ti \{0\}$ and then glue the two boundary components together
via $h$ instead of the identity. This gives a new example $f$
defined on the mapping torus $M_h$ instead of $M \ti S^1.$
Theorems \ref{thm:anosovswitch} and \ref{thm:multiswitch} below
can also be generalized in a similar manner.
However, we do not consider this type of generalization in any further detail in the
current paper.

\medskip{}

The definition of partial hyperbolicity we are using in this paper
is sometimes called ``pointwise'' partial hyperbolicity,
in contrast to a stricter notion of absolute partial hyperbolicity.
A partially hyperbolic diffeomorphism $f : M \to M$ is
\emph{absolutely partially hyperbolic} if there are global constants
$\eta < 1 < \mu$ such that 
\[
    \| Df^k v^s \| < \eta < \| Df^k v^c \| < \mu < \| Df^k v^u \|
\]
for all $p \in M$ and unit vectors
$v^s \in \Es_p$, $v^c \in \Ec_p$, and $v^u \in \Eu_p$.
Again, we may freely assume that $k = 1.$

% Note: Paragraph break removed to get theorem to fit on one page.
 
For partially hyperbolic diffeomorphisms on the 3-torus,
this distinction between ``pointwise'' and ``absolute'' is important.
Any absolutely partially hyperbolic diffeomorphism on $\bbT^3$
is dynamically coherent \cite{BBI2}
and bundle switching cannot be used in this setting.

In higher dimensions,
the techniques used in the proof of \cref{thm:oneswitch} can be adapted
to construct absolutely partially hyperbolic
examples which are dynamically incoherent.
To do this, the construction uses Anosov flows.
Let $\phi$ be a flow generated by a $C^1$ vector field $X$ on a manifold $M.$
Such a flow $\phi$ is an \emph{Anosov flow}
if for any $t > 0,\ \phi_t$ is partially hyperbolic with
a one-dimensional center direction $\Ecphi$ given by the span of $X.$
An Anosov flow may be thought of as a parameterized family of partially
hyperbolic maps $\phi_t$ where the amount of hyperbolicity can be
dialled up and down by adjusting the parameter $t \in \bbR.$
This allows us to modify the proof of \cref{thm:oneswitch}
to prove the following.

\begin{thm} \label{thm:anosovswitch}
    Let $\phi_t : M \to M$ be an Anosov flow
    with a one-dimensional oriented strong stable bundle, $\Esphi.$
    Then there is an absolutely partially hyperbolic diffeomorphism
    $f : M \ti S^1 \to M \ti S^1$
    such that
    \begin{enumerate}
        \item $M \ti \{0\}$ is a cu-submanifold,
        \item
        $f(x,0) = (\phi_1(x),0)$ for all $x \in M,$ and
        \item
        $f$ is dynamically incoherent.
    \end{enumerate} \end{thm}
Finally,
we establish a version of \cref{thm:oneswitch} which holds for
a higher-dimen\-sional stable bundle $\Esg,$
so long as the bundle splits into one-dimensional subbundles.

\begin{thm} \label{thm:multiswitch}
    Suppose that $g : M \to M$ is a partially hyperbolic diffeomorphism such that
    the stable bundle has an invariant dominated subsplitting
    \[
        \Esg = \Es_1 \oplus \Es_2 \oplus \cdots \oplus \Es_d
    \]
    into one-dimensional subbundles
    and that $g$ preserves the orientation of each subbundle $\Es_i.$
    Then 
    there is a partially hyperbolic diffeomorphism
    $f : M \ti \bbT^d \to M \ti \bbT^d$ such that
    \begin{enumerate}
        \item $M \ti \{0\}$ is a cu-submanifold,
        \item
        $f(x,0) = (g(x), 0)$ for all $x \in M,$ and
        \item
        $f$ is dynamically incoherent.
    \end{enumerate} \end{thm}
In both theorem \ref{thm:oneswitch} and \ref{thm:multiswitch},
the stable bundle of $g$ becomes part of the center bundle of $f$
on the \emph{cu}-submanifold $M \ti \{0\},$ and so the dimensions of the
splitting are
\[
    \dim \Esf = \dim \Esg,
    \quad
    \dim \Ecf = \dim \Ecsg,
    \qandq
    \dim \Euf = \dim \Eug.
\]
The same holds for \cref{thm:anosovswitch} with $\phi$ in place of $g$

\medskip

In related work,
the papers \cite{BGHP,ham2018construct}
also give constructions of dynamically incoherent systems.
The work in \cite{BGHP} uses a technique called ``$h$-transversality''
to build examples on unit tangent bundles of surfaces
and these examples can be made
absolutely partially hyperbolic,
robustly transitive,
and stably ergodic.
The $h$-transversality method is quite different to
the bundle switching used here.

There are two different constructions in \cite{ham2018construct}
and both may be thought of as bundle switching.
The first is on the 3-torus and shows that the original construction
of \cite{RHRHU-nondyn} can be modified so that the center-unstable torus
has derived-from-Anosov dynamics instead of Anosov dynamics.

The second construction, in  \cite[\S 7]{ham2018construct},
produces high-dimensional examples on manifolds of the form
$M \ti \bbT^D.$
The base manifold $M$ can be anything, but the dynamics on the fibers $\bbT^D$
is quite restricted.
It is a product $A \ti \cdots \ti A$ of $D/2$ copies of
a linear Anosov map $A : \bbT^2 \to \bbT^2.$
This is the opposite case to \cref{thm:multiswitch},
where the base manifold is $\bbT^d,$ but the fiber dynamics
can be a partially hyperbolic map on a general manifold.
Moreover, \cref{thm:multiswitch} proves dynamical incoherence
whereas \cite[\S 7]{ham2018construct} only establishes the presence of
embedded compact submanifolds tangent to $\Ecs$ or $\Ecu.$
Both constructions in \cite{ham2018construct} rely on the fact that
the stable bundle in the fibers is linear and use linear shearings.

\medskip{}

There were several motivations for the work in the current paper.

One motivation was to assess the difficulty of classifying partially
hyperbolic systems in dimensions 4 and higher.
There are many recent classification results for
3-dimensional partially hyperbolic systems
in a variety of settings, though a complete classification has not yet been
achieved. Many of the results were first established in the dynamically
coherent setting \cite{HP2, 1908.06227}
and then extended to non-dynamically coherent systems
\cite{hp2019class, 2008.04871}.
For instance, the classification of partially hyperbolic systems
on the 3-torus is more complicated when
\emph{cs} and \emph{cu}-tori are present.
Fortunately, only a small family of 3-dimensional manifolds support
partially hyperbolic systems with \emph{cs} and \emph{cu}-tori.
The work in the current paper shows that in dimension 4 and higher,
partially hyperbolic systems with
compact \emph{cs} and \emph{cu}-sub\-manifolds
are not limited to a small family of manifolds, and so they will 
significantly complicate attempts at classifying these systems.

The constructions given in this paper add at least one
to the dimension of the center bundle.
For instance, if dim $\Ecg = 1$ in \cref{thm:oneswitch},
then dim $\Ecf = 2$ in the constructed system.
Therefore, there is some hope that classification of systems with
one-dimensional center may still be tractable in the higher-dimensional
setting.
The proofs of theorems \ref{thm:oneswitch} and \ref{thm:multiswitch}
can be adapted to the case where dim $\Ecg = 0;$
that is, where $g$ is an Anosov diffeomorphism.
This means, in particular, that we can construct a dynamically incoherent
partially hyperbolic diffeomorphism $f : M \ti S^1 \to M \ti S^1$
with one-dimensional center using an Anosov diffeomorphism $g : M \to M$
with one-dimensional stable direction.
Of course, such an Anosov diffeomorphism must be defined on a torus
$M = \bbT^d$ \cite{newhouse1970codimension}.
Perhaps it might be the case that such systems only exist on
a limited number of manifolds, analogous to the 3-dimensional results in
\cite{RHRHU-tori}.

\begin{question}
    Which manifolds in dimension 4 and higher
    support a partially hyperbolic system with
    one-dimensional center and with a compact cs or cu-submanifold?
\end{question}
An additional motivation is the goal of building transitive examples with some
form of 
bundle switching. For instance, could one build a transitive partially
hyperbolic skew product where some of the fibers are tangent to $\Es \oplus \Eu$
and other fibers are tangent to $\Ec?$

By applying \cref{thm:oneswitch} to a system
$g : \bbT^3 \to \bbT^3,$
with a \emph{cs}-torus,
we can produce a (non-transitive) partially hyperbolic diffeomorphism
$f$ defined on the 4-torus $\bbT^2 \oplus \bbT^2,$
regarded as a skew product with fibers of the form $\{x\} \ti \bbT^2.$
The two-dimensional center direction $\Ecf$
is tangent to a single fiber $\{x_0\} \ti \bbT^2,$
but transverse to the fibers everywhere except a
3-dimensional \emph{cu}-torus.
This example is not transitive, but if some form of transverse version
of this construction is possible, it would likely be similar to this example
in a neighbourhood of $\{x_0\} \ti \bbT^2.$

Thus far, all of the bundle switching constructions, both here and in
\cite{RHRHU-nondyn, ham2018construct}
are non-transitive, building either an attracting \emph{cu}-submanifold
or repelling \emph{cs}-submanifold.
Hopefully, generalizing these techniques as far as possible
will lead to transitive constructions.

A further motivation is to develop these techniques in the partially
hyperbolic setting as far as possible in order to see if they might also be
applied in the Anosov setting. It is a long standing open question in dynamics
whether every Anosov diffeomorphism is transitive.
If non-transitive examples of Anosov diffeomorphisms actually do exist,
it is conceivable that a (much more advanced) version of bundle switching
might be used to construct examples of such systems.

\medskip{}

The paper first gives a complete proof of \cref{thm:oneswitch}
and then later explains how the proof can be adapted to establish
\cref{thm:anosovswitch} and \cref{thm:multiswitch}.
This means regrettably that later sections list out changes to make to
``patch'' earlier sections. However, we felt it best to present the key ideas
of the construction in the simplest setting possible, before treating more
general cases.

\Cref{sec:construction} constructs the partially hyperbolic diffeomorphism $f$ in
\cref{thm:oneswitch} and shows it satisfies items (1) and (2) of that theorem.
\Cref{sec:incoherent} then establishes item (3), dynamical incoherence.
\Cref{sec:flows} explains how to adapt these proofs to show \cref{thm:anosovswitch}.
\Cref{sec:high} states and proves a result, \cref{prop:highswitch},
which holds for a higher-dimensional stable bundle.
Then, \cref{sec:multi} shows how this proposition can be used to prove
\cref{thm:multiswitch}.

\section{Construction} \label{sec:construction} %{{{1

To establish the existence of dominated splittings,
we use the following result given in \cite{ham2018construct}.

\begin{thm} \label{thm:genchaindom}
    Suppose $f$ is a diffeomorphism of a manifold $M$
    and $Y$ and $Z$ are compact invariant subsets such that
    \begin{enumerate}
        \item all chain recurrent points of $f|_Y$ lie in $Z,$
        \item
        $Z$ has a dominated splitting $T_Z M = \Eu \oplus \Es$
        with $d = \dim \Eu,$
        and
        \item
        for every $x \in Y \sans Z,$
        there is a point $y$ in the orbit of $x$
        and a subspace $V_y$ of $T_y M$ of dimension $d$ such that
        for any non-zero $v \in V_y,$ each of the sequences
        \[
            \left\{ \frac{Df^n v}{ \| Df^n v \| } \right\}_{n=0}^{\infty}
            \qandq
            \left\{ \frac{Df^{-n} v}{ \| Df^{-n} v \| } \right\}_{n=0}^{\infty}
        \]
        accumulates on a vector
        in $T_Z M \sans \Es$ as $n \to +\infty$.
    \end{enumerate}
    Then the dominated splitting on $Z$ extends to a dominated splitting
    on $Y \cup Z.$
\end{thm}
As noted in \cite{ham2018construct}, the subspace $V_y$ will be equal to $\Eu(y)$
in the extended dominated splitting.
Even using this theorem,
we will still deal extensively with cone families and other subsets
of the tangent bundle. Because of this, we introduce notation tailored for
working with these subsets.
\begin{notation}
    Suppose $\Aone$ and $\Bone$ are subsets of $TM.$
    \begin{enumerate}
        \item Let $\Aone(x)$ denote $\Aone \cap T_x M.$
        \item
        Write $\Aone \cc \Bone$ to signify that for each $x \in M,$ any non-zero
        vector $v$ in the closure of $\Aone(x)$ is
        contained in the interior of $\Bone(x).$
        \item
        Define the pointwise sum
        \begin{math}
            \displaystyle % inline math to avoid page break inside list
            \Aone + \Bone = \bigcup_{x \in M} \big(\Aone(x) + \Bone(x) \big).
        \end{math}

        \noindent
        That is, $u \in \Aone + \Bone$ if and only if there
        are $x \in M, v \in \Aone(x),$ and $w \in \Bone(x)$ such that $u = v + w.$
        \item
        Define $-\Aone$ by
        $v \in -\Aone$ if and only if $-v \in \Aone.$
        \item
        Define $\Aone - \Bone$ as $\Aone + (-\Bone).$
    \end{enumerate} \end{notation}
As in \cite{cropot2015lecture},
a \emph{cone family} is a subset of $TM$ defined by
\[
    \{ v \in TM : Q(v) \ge 0 \}
\]
where $Q : TM \to \bbR$ is a continuous function which restricts to a quadratic
form $Q|_{T_x M}$ on each tangent space. We assume $M$ is connected and so
the quadratic forms all have the same signature independent of $x,$
and this defines the \emph{dimension} of the cone family.

To construct an example on $M \ti S^1,$
we first define a diffeomorphism $f$ from $M \ti [0,1]$ to itself.
Then, we extend $f$ to $M \ti [-1,1]$ by the symmetry
\[
    f(x, -z) = (x_1, -z_1)
    \quad \text{where} \quad
    (x_1, z_1) = f(x, z).
\]
Finally, we identify the endpoints of the interval $[-1,1]$
to produce a diffeomorphism of
\begin{math}
    M \ti ([-1,1] / \sim ) = M \ti S^1.
\end{math}

Consider $g : M \to M$ partially hyperbolic with
one-dimensional stable direction $\Esg.$
As in \cref{thm:oneswitch}, assume that $g$ preserves an orientation of $\Esg.$
Fix constants $0 < \lam < \eta < 1 < \mu$ such that
\[
    \lam < \| Dg v^s \| < \eta \qandq \mu < \| Dg v^u \|
\]
for all unit vectors $v^s \in \Esg$ and $v^u \in \Eug.$
(Later in \cref{sec:flows}, the constant $\lam$ will be chosen
differently.)
Note that we do not assume that $g$ is absolutely partially hyperbolic,
so it may be the case that there are unit vectors $v^c \in \Ecg$
such that $ \| Dg v^c \| $ lies outside of the interval $[\eta, \mu].$

\begin{figure}
    \centering
    \includegraphics{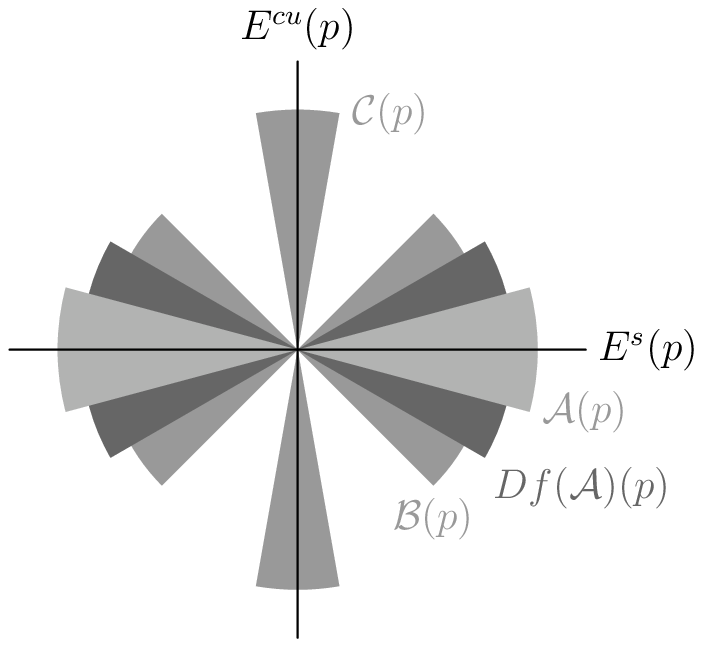}
    \caption{
    A depiction of the cone families, $\Aone, \Bone,$ and $\Cone$
    at a point $p.$ For simplicity, we draw $\Ecu$ as if it were one-dimen\-sional.}
    \label{fig:cones}
\end{figure}
For the partially hyperbolic splitting
\[
    TM = \Esg \oplus \Ecg \oplus \Eug,
\]
let $\Aone$ be a cone family associated to $\Esg.$
That is, $\Esg \cc \Aone \cc Dg(\Aone)$ and the dual cone family
$\Aone^* = \overline{TM \sans \Aone}$
satisfies $\Ecug \cc \Aone^*.$
Use 0 to denote the zero section of the tangent bundle $TM.$
By assumption $\Esg$ is one-dimensional and oriented, and so
$\Esg \sans 0$ has two connected components
$E^+$ and $E^-.$
Here, $E^+$ is the component consisting of all vectors pointing
in the positive direction of $\Esg.$
Further,
$\Aone \sans 0$ has two connected components
$\Aone^+$ and $\Aone^-$
where $E^+$ is a subset of $\Aone^+.$
Let $\Bone = Dg^2(\Aone)$ and
define $\Bone^+$ and $\Bone^-$ analogously to $\Aone^+$ and $\Aone^-.$
Since by assumption $Dg$ preserves the orientation of $\Esg,$
\[
    \Aone^+ \cc Dg(\Aone^+) \cc \Bone^+ \cc Dg(\Bone^+).
\]
Using the dual cone, define
$\Cone = Dg(\Bone^*)$
so that
$\Ecug \cc Dg(\Cone) \cc \Cone.$
%Bone cap Cone = 0.
Note that $\Bone \cap \Cone$ is the zero section of $TM$
and $\Bone^+ \cap \Cone$ is the empty set.
See \cref{fig:cones}.
Finally, define an unstable cone family $\Uone$ associated
to the dominated splitting $\Ecsg \oplus \Eug.$
That is, $\Eug \cc Df(\Uone) \cc \Uone$ and $\Ecsg \cc \Uone^*.$

\begin{figure}

\centering
\includegraphics{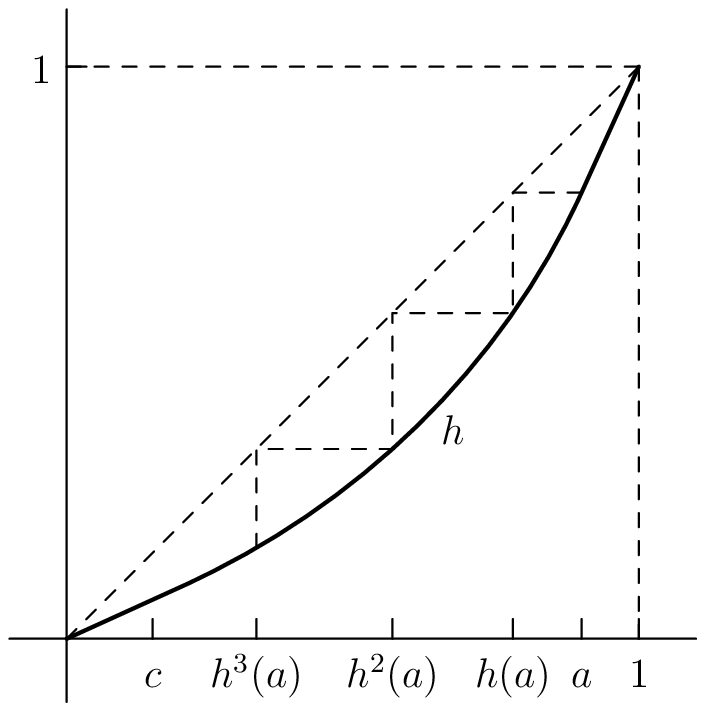}
\includegraphics{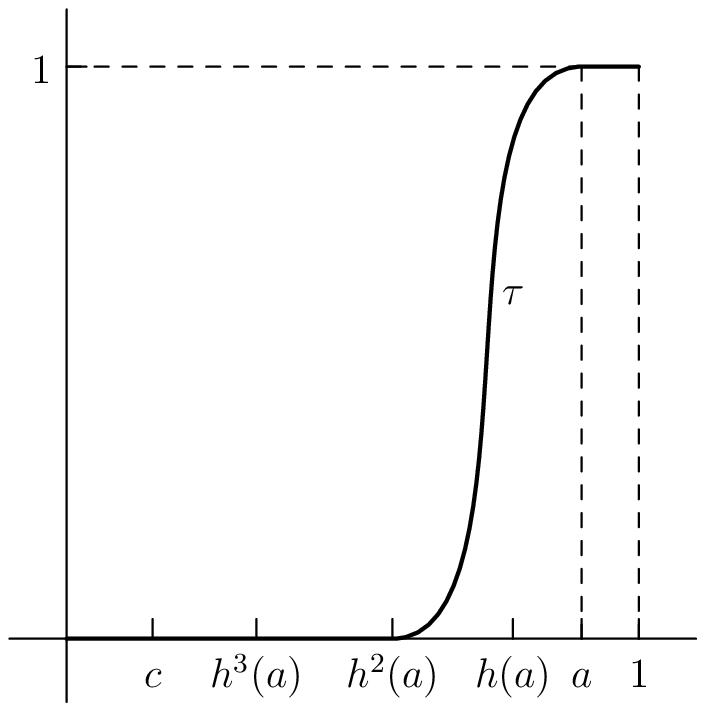}
\caption{The graphs of the functions $h$ and $\tau$.}
\label{fig:htau}
\end{figure}
We now define functions $h$ and $\tau$ as depicted in figure \ref{fig:htau}.
Fix constants $0 < c < a < 1$ and
define a smooth function $h : [0,1] \to [0,1]$ such that
the following properties hold:
\begin{enumerate}
    \item $h(z) < z$ for all $0 < z < 1;$
    \item
    $h(z) = 1 - \mu \cdot (1 - z)$ for all $a < z < 1;$ and
    \item
    $h(z) = \lam \cdot z$ for all $0 < z < c.$
\end{enumerate}
Note that $h(0) = 0$ and $h(1) = 1.$ We assume that $c < h^3(a).$
Define a smooth function $\tau : [0,1] \to [0,1]$ with
the following properties:
\begin{enumerate}
    \item $\tau(z) = 0$ for all $z \le h^2(a);$
    \item
    $\tau'(z) > 0$ for all $h^2(a) < z < a;$ and
    \item
    $\tau(z) = 1$ for all $z \ge a.$
\end{enumerate}
In general, the subbundle $\Esg$ will not be $C^1,$
but it can be smoothly approximated.
Therefore, choose a smooth vector field
$X$ on $M$ such that $X \subof \Aone^+.$
This defines a smooth flow $\sig_t$ on $M.$
Write $g_t : M \to M$ for the composition $g_t = \sig_t \circ g.$
If $t$ is sufficiently small, then $\sig_t$ is $C^1$ close to the identity
map on $M$ and so the inclusions
\[
    \Aone \cc Dg_t(\Aone) \cc \Bone \cc Dg_t(\Bone),
    \quad
    Dg_t(\Cone) \cc \Cone,
    \qandq
    Dg_t(\Uone) \cc \Uone
\]
all hold.
By rescaling the vector field $X,$ we can assume
without loss of generality that these inclusions hold
for all $g_t$ with $|t| \le 1.$
It follows that $g_1$ has a partially hyperbolic splitting
$TM = E^s_{g_1} \oplus E^c_{g_1} \oplus E^u_{g_1}$ with
\[
    E^s_{g_1} \cc \Aone,
    \quad
    E^{c u}_{g_1} \cc \Cone,
    \qandq
    E^u_{g_1} \cc \Uone.
\]
    
Let $I$ denote the interval $[0,1]$
and define the diffeomorphism
\[
    f : M \ti I \to M \ti I,
    \quad
    (x, z) \mapsto ( g_{\tau(z)}(x), h(z) ).
\]
We will show that this is partially hyperbolic
with
\[
    \dim \Esf = \dim \Esg,
    \quad
    \dim \Ecf = \dim \Ecg + 1,
    \qandq
    \dim \Euf = \dim \Eug.
\]
In order to analyze the derivative of $f,$
for each $z \in [0,1]$ we identify its tangent
space $T_z I$ with $\bbR$ in the obvious way.
Then we can write a vector in
$T(M \ti I) = TM \ti TI$
in the form $(u, v)$ where
$u \in TM$ is the ``horizontal'' part of the vector and
$v \in \bbR$ is the ``vertical'' part of the vector.
If this tangent vector is based at a point
$(x, z) \in M \ti I,$ then the derivative of $f$ satisfies
\[
    Df(u, v) \ =\ 
    \big( Dg_{\tau(z)} (u)\ +\ v \cdot \tau'(z) \cdot X(y),\ \
    v \cdot h'(z) \big)
\]
where
$Dg_t : TM \to TM$ is the derivative of $g_t : M \to M$ for a fixed $t,$ 
$h'$ is the usual single-variable calculus
notion of a derivative of the map $h : I \to \bbR,$
and $y = g_{\tau(z)}(x).$
On the set $M \ti \{0\}, f$ has a partially hyperbolic splitting
\[
    \Esf(x,0) = 0 \ti \bbR,
    \quad
    \Ecf(x,0) = \Ecsg(x) \ti 0,
    \qandq
    \Euf(x,0) = \Eug(x) \ti 0.
\]
On the set $M \ti \{1\}, f$ has a partially hyperbolic splitting
\[
    \Esf(x,1) = E^s_{g_1}(x) \ti 0,
    \quad
    \Ecf(x,1) = E^c_{g_1}(x) \ti \bbR,
    \qandq
    \Euf(x,1) = E^u_{g_1}(x), \ti 0.
\]
We use \cref{thm:genchaindom} twice
to extend this splitting to all of $M \ti [0,1].$

\begin{notation}
    For the remainder of the paper, we adopt the following notation.
    If $p = (x, z)$ is a point in $M \ti I$ and
    $(u, v) \in T_x M \ti \bbR$ is a tangent vector based at $p,$
    then for all $n \in \bbZ:$
    \[
        (x_n, z_n) = f^n(x, z),
        \quad
        t_n = \tau(z_n),
        \quad
        (u_n, v_n) = Df^n(u, v),
        \qandq
        w^n = \frac{(u_n, v_n)}{ \| (u_n, v_n) \| }.
    \]
    In most places where we use this notation,
    $p$ will be a point in the fundamental domain
    $M \ti (h(a),a].$
\end{notation}
\begin{lemma} \label{lemma:extendcu}
    The dominated splitting $\Esf \oplus_< \Ecuf$ on $M \ti \{0,1\}$
    extends to all of $M \ti [0,1].$
\end{lemma}
\begin{proof}
    For a point $p = (x,z) \in M \ti (h(a),a],$ define
    a subspace $W$ of the tangent space by
    \[
        W = \Ecu_{g_1}(x) \ti \bbR \subof T_x M \ti \bbR.
    \]
    Let $(u, v)$ be a non-zero vector in $W.$
    First consider the behaviour as $n \to -\infty.$
    Since $\tau(z_n) = 1$ and $\tau'(z_n) = 0$ for all $n < 0,$
    it follows that $u_n \in \Ecu_{g_1}$ for all $n < 0.$
    If a subsequence $\{n_j\}$ is such that $\lim_{j \to \infty} n_j = -\infty$
    and $\{w^{n_j}\}$ converges to tangent vector based at a point on
    $M \ti \{1\},$
    then it must be that $\lim w^{n_j} \in \Ecuf.$

    Now consider the behaviour as $n \to +\infty.$
    Replacing $(u, v)$ by $(-u, -v)$ if necessary,
    we can freely assume that $v \ge 0.$
    We first consider the case where both $v > 0$ and $z < a$ hold
    and show that $u_2$ is non-zero.
    Then we show the same result for the special cases of
    $v = 0$ and $z = a.$
    From the formula for the derivative $Df,$
    \[
        (u_1, v_1)
        \ =\
        \big(Dg_{t_0}(u_0)\ + \ v_0 \cdot \tau'(z_0) \cdot X(x_1),\ \
        v_0 \cdot h'(z_0) \big).
    \]
    Write $u_1 = u_C + u_A$ where
    \[
        u_C = Dg_{t_0}(u_0) \in \Cone
        \qandq
        u_A = v_0 \cdot \tau'(z_0) \cdot X(x_1) \in \Aone^+.
    \]
    Then
    \[
        (u_2, v_2)
        \ =\
        \big(Dg_{t_1}(u_C)\ +\ Dg_{t_1}(u_A)\ +\
        v_1 \cdot \tau'(z_1) \cdot X(x_2),\ \ v_1 \cdot h'(z_1) \big).
    \]
    Note that
    \[
        Dg_{t_1}(u_C) \in \Cone,
        \quad
        Dg_{t_1}(u_A) \in \Bone^+,
        \qandq
        v_1 \cdot \tau'(z_1) \cdot X(x_2) \in \Aone^+.
    \]
    The sum of a vector in $\Bone^+$ with a vector in $\Aone^+$
    is a vector in $\Bone^+.$
    In particular, it is non-zero and does not lie in $\Cone.$
    Therefore $u_2$ is non-zero as it is the sum of three vectors in
    $\Cone, \Bone^+,$ and $\Aone^+$ respectively.

    In the special case where $z = a,$
    $\tau'(a) = 0$ implies that $u_A = 0$ and therefore
    $u_2 \in \Cone + \Aone^+$ is non-zero.
    In the special case where $v = 0,$
    $u_0$ is non-zero and $g_{t_1} \circ g_{t_0}$ is a diffeomorphism, so
    $u_2 = D(g_{t_1} \circ g_{t_0})(u_0)$ is non-zero.

    Since $z_2 = h^2(z) \le h^2(a),$ it follows that
    $\tau(z_n) = 0$ and $\tau'(z_n) = 0$ for all $n \ge 2.$
    Hence $u_n = Dg^{n-2}(u_2)$ for all $n \ge 2.$
    Since the set of unit vectors in $\Esg$ is compact,
    there is $\delta > 0$ such that
    $ \| Dg|_{\Esg(x)} \| \ge \lam + \delta$
    holds for all $x \in M.$ Therefore
    \[
        \| u_{n+1} \| \ge (\lam + \delta) \| u_n \|
        \qandq
        |v_{n+1}| \le \lam |v_n|
    \]
    hold for all large positive $n.$
    We are considering a subsequence where $\lim n_j = +\infty$
    and $\{w^{n_j}\}$ converges to a tangent vector based at a point
    in $M \ti \{0\}.$ The above inequalities imply that
    $\lim w^{n_j} \in TM \ti 0.$
    The hypotheses of \cref{thm:genchaindom} are verified and
    the dominated splitting
    $\Esf \oplus \Ecuf$ on $M \ti \{0,1\}$
    extends to a dominated splitting on $M \ti [0,1].$
\end{proof}

\begin{notation}
    For later proofs, we introduce the notation
    \[
        u_B = Dg_{t_1}(u_A) + v_1 \cdot \tau'(z_1) \cdot X(x_1),
    \]
    so that in the setting of \cref{lemma:extendcu},
    $u_2 = Dg_{t_1}(u_C) + u_B$
    with $Dg_{t_1}(u_C) \in \Cone$ and $u_B \in \Bone^+.$
\end{notation}
%FLOWS
%
%Combining u_2 ∈ u_B + Cone(x_2) with u_3 ∈ D phi_r(u_2) + Ecphi(x_3)
%and using
%Ecg(x_3) = D phi_r( Ecg(x_2) )
%yields
%
%    u_3\ ∈\ D phi_r big(u_B + Cone(x_2) + Ecg(x_2) big)
%        \ =\ D phi_r big(u_B + Cone(x_2) big).
%
%Since
%D phi_r : T_{x_2} M -> T_{x_3} M is a linear isomorphism
%and any element of u_B + Cone(x_2) is non-zero,
%u_3 is non-zero.
%The rest of the proof then follows as before.
%
%HIGHER DIM
%
%The original proof (using the new cone families)
%shows that u_2 is a sum of a vector in Bone^+
%with a vector in Cone, and therefore u_2 notin Es_2
%by the above assumption.
%Consequently as n_j -> +infty, w^{n_j} converges to a vector
%outside of Es_2 ti bbR.
%
%Note that this adaptation of the proof establishes an additional 
%property that we will need to use later:
%quote:
%    for all (x, z) ∈ M ti (0, h^2(a)],
%    Ecuf(x, z) is transverse to Es_2(x) ti bbR.
%
%FLOWS AND HIGHER DIM
%
%For any t > 0, D phi_t leaves Es_2 invariant.
%Since u_B + Cone(x_2) cap Es_2(x_2) = varnothing,
%this implies that
%D phi_r(u_B + Cone(x_2)) cap Es_2(x_3) = varnothing
%and so u_3 notin Es_2.
%For n >= 3, u_{n+1} = D phi_1(u_n) notin Es_2.
%Then as n_j -> +infty, w^{n_j} converges to a vector
%outside of Es_2 ti bbR.

\begin{lemma} \label{lemma:extendu}
    The dominated splitting $\Ecsf \oplus \Euf$ on $M \ti \{0,1\}$
    extends to a dominated splitting on $M \ti [0,1].$
\end{lemma}
\begin{proof}
    For a point $(x, z)$
    in the fundamental domain $M \ti (h(a), a],$
    define the subspace
    \begin{math}
        U = \Eu_{g_1}(x) \ti 0 \subof T_x M \ti \bbR
    \end{math}
    and consider a non-zero
    vector $(u,v) = (u, 0) \in U.$
    First consider the behaviour as $n \to -\infty.$
    Since
    $\tau(z_n) = 1$ and $\tau'(z_n) = 0$ for all $n < 0,$
    \[
        (u_n, v_n) = (Dg_1^n(u_n), 0) \in \Eu_{g_1}(x_n) \ti 0
    \]
    for all $n < 0.$
    If $\lim n_j = -\infty$ and $\{w^{n_j}\}$ converges to
    a tangent vector based at a point in $M \ti \{1\},$
    then $\lim w^{n_j}$ lies in $E^u_{g_1} \ti 0.$

    Now consider the behaviour as $n \to +\infty.$
    Since $v = 0$ and $Dg_t \Uone \cc \Uone$ for all $|t| \le 1,$
    it is easy to see
    that $u_n \in \Uone$ and $v_n = 0$ for all $n \ge 0.$
    If $\lim n_j = +\infty$ and $\{w^{n_j}\}$ converges to
    a tangent vector based at a point in $M \ti \{0\},$
    then $\lim w^{n_j}$ lies in $\Uone \ti 0.$
    In particular, the limit vector is not in $\Ecsg \ti \bbR.$
    Hence, the hypotheses of \cref{thm:genchaindom} are verified.
\end{proof}
%FLOWS
%
%The fact that u_n ∈ Uone for all n > 0
%now relies on the property that
%D phi_t(Uone) subof Uone for all t > 0,
%but otherwise the proof is unchanged.
%
%HIGHER DIM
%
%No change.

We have now established a partially hyperbolic splitting on all of
$M \ti [0,1].$
To prove dynamical incoherence in the next section,
we need further information about
the center-unstable bundle $\Ecuf.$

\begin{lemma} \label{lemma:trans}
    For all $(x, z) \in M \ti (0, h^2(a)],$
    \begin{enumerate}
        \item the intersection
        $\Ecuf(x, z) \cap (T_x M \ti 0)$
        is a subset of $\Cone(x) \ti 0$
        \item
        $\Ecuf(x, z)$ is transverse to $\Esg(x) \ti \bbR,$ and
        \item
        if $(u, v) \in \Ecuf(x, z)$ with $u \in \Esg(x)$ and $u$ is non-zero,
        then $v$ is non-zero.
    \end{enumerate} \end{lemma}
\begin{remark}
    The transverse intersection in item (2) of \cref{lemma:trans}
    is one-dimensional.
    This will hold true even in the setting of \cref{sec:high},
    where we drop the restriction that
    dim $\Esg = 1.$
\end{remark}
\begin{proof}
    Instead of proving this for a point in $M \ti (0, h^2(a)],$
    we consider $(x, z)$ in the fundamental domain $M \ti (h(a), a]$ and
    establish the results for $(x_n, z_n)$ where $n \ge 2.$
    Note that $\Ecuf(x_n, z_n) = Df^n(W)$
    where $W$ is the subspace defined in the proof of \cref{lemma:extendcu}.
    For each $n \ge 0,$ define a subspace $Y_n$ of $T_{x_n} M$ by
    \[
        Y_n \ti 0 = \Ecuf(x_n, z_n) \cap (T_{x_n} M \ti 0).
    \]
    From the definition of $f,$
    observe that
    $Y_{n+1} \ti 0 = Df(Y_n \ti 0)$ for all $n \ge 0.$
    Then
    \[  Y_0 = \Ecu_{g_1} \subof \Cone
        \qandq Y_{n+1} = Dg_{t_n}(Y_n)\ \text{for all $n \ge 0.$}  \]
    It follows by induction that $Y_n \in \Cone$ for all $n \ge 0$
    and establishes item (1).
    Since
    $\Cone(x_n) \cap \Esg(x_n) = 0,$
    this further implies that 
    $Y_n \cap \Esg(x_n) = 0.$
    Using $\dim(Y_n) = \dim(\Ecug),$ we can show by counting dimensions that
    $Y_n \oplus \Esg(x_n) = T_{x_n} M$
    and from this fact items (2) and (3) follow.
\end{proof}    %Y_n cap Esg(x_n) = 0.

%FLOWS
%
%In the flow case,
%Y_3 = D phi_r(Y_2) subof Cone
%and
%Y_{n+1} = D phi_1(Y_n)
%for n >= 3.
%Then as D phi_t (Cone) subof Cone for any t > 0,
%it follows by induction that Y_n ∈ Cone for all n >= 0
%and the rest of the proof is as before.
%
%HIGHER DIM 
%
%Lemma:trans and its proof hold without any changes in the high-dim case.
%However, here we can actually establish the additional property that
%Ecuf(x, z) is transverse to Es_2(x) ti 0.

For the next result, recall that $E^+$ denotes
the set of tangent vectors pointing in the positive direction of $\Esg.$

\begin{lemma} \label{lemma:allpos}
    Consider a point 
    $(x, z) \in M \ti (0, h^2(a)]$ and a vector
    $(u,v) \in \Ecuf(x, z)$ such that $u \in \Esg(x).$
    Then
    $u \in E^+$
    if and only if $v > 0.$
\end{lemma}
\begin{proof}
    First define a function
    $\hat u : M \ti (0, h^2(a)] \to \Esg \sans 0$
    as follows:
    for each point $(x, z) \in M \ti (0, h^2(a)],$
    let $\hat u(x, z)$ be the unique $u \in \Esg(x)$
    such that $(u, 1) \in \Ecuf(x, z).$
    \Cref{lemma:trans} shows that such a vector exists, is unique, and is non-zero.
    By continuity, 
    $\hat u$ 
    only takes values in a single connected component
    of $\Esg \sans 0,$
    and so to prove the lemma
    it suffices to show that $\hat u(x, z) \in E^+$ for a single point in
    $M \ti (0, h^2(a)].$

    Consider a point $(x_2, z_2)$ in the subset $M \ti (h^3(a), h^2(a))$
    and define
    \[    
        u_2 = \hat u(x_2, z_2) \in \Esg(x_2) \qandq v_2 = 1
    \]        
    so that $(u_2, v_2) \in \Ecuf(x_2, z_2).$
    From these, define
    \[
        (x, z) = f^{-2}(x_2, z_2)
        \qandq
        (u, v) =  Df^{-2}(u_2, v_2).
    \]
    It follows that 
    $(x, z) \in M \ti (h(a), a),$ $(u, v) \in \Ecuf(x, z),$ and $v > 0.$
    Recall from the remark after
    the proof of \cref{lemma:extendcu}
    that
    $u_2 = Dg_{t_1}(u_C) + u_B$
    with $Dg_{t_1}(u_C) \in \Cone$ and $u_B \in \Bone^+.$
    If $u_2 \in E^-,$ then $-u_2 \in E^+ \subof \Bone^+$ and
    \[
        u_B - u_2 = - Dg_{t_1}(u_C) \in \Bone^+ \cap \Cone,
    \]
    a contradiction since 
    $\Bone^+$ and $\Cone$ are disjoint.
    Therefore, $u_2 \in E^+.$
\end{proof}
%FLOWS
%
%No change.
%
%HIGHER DIM
%
%This is given in the section.

As explained at the start of the section,
we extend $f$ to $M \ti [-1,1]$ by symmetry,
and then identify the boundary components to produce an example on $M \ti S^1.$

\section{Dynamical incoherence} \label{sec:incoherent} %{{{1

The previous section constructed a partially hyperbolic map
which satisfies all but the last item listed in \cref{thm:oneswitch}.
We now show that this construction
is dynamically incoherent in order complete the proof of the theorem.

In the paper \cite{RHRHU-nondyn}
giving the original dynamically incoherent example on $\bbT^3,$
they explicitly construct the center direction $\Ecf,$
and it is clear from the construction that there cannot be
a foliation tangent to it.
For the examples here, the center is at least two-dimensional
and the dynamics on the attractor $M \ti \{0\}$ can be
a general partially hyperbolic diffeomorphism
(including both dynamically coherent and incoherent cases)
and so we give here a detailed and rigorous proof that the overall system is
dynamically incoherent.

On $M \ti I,$
we denote the coordinate projections by
$\pi_M : M \ti I \to M$
and
$\pi_I : M \ti I \to I.$
For a curve $\gam : [0,1] \to M \ti I,$ we adopt the notation
$\gam_M = \pi_M \circ \gam: [0,1] \to M$
and
$\gam_I = \pi_I \circ \gam: [0,1] \to I$
for the projected curves.

Recall by assumption that $\Esg$ is one-dimensional and oriented.
Therefore any non-zero vector in $\Esg$ will
point in either the positive $E^+$ or negative $E^-$ direction of $\Esg.$
Also recall the constant $c > 0$ used in defining the function $h.$

A $C^1$ curve $\gam : [0,1] \to M \ti (0,c]$ is called a
\emph{falling curve}
if for all $t \in [0,1]$:
\begin{enumerate}
    \item the vector $\gam'(t) \in T(M \ti I)$
    is tangent to $\Ecuf,$ and
    \item
    the vector $\gam_M'(t) \in TM$
    is in $E^-.$
\end{enumerate}    

%HIGHER DIM
%
%enum:
%    the vector gam'(t) ∈ T(M ti I)
%    is tangent to Ecuf, and
%    ---
%    the vector gam_M'(t) ∈ TM
%    is in Eone^-.

We first observe that $\gam_I$ is a decreasing function,
which justifies the name ``falling.''

\begin{lemma} \label{lemma:falling}
    If $\gam$ is a falling curve, then
    $\gam_I'(t) < 0$
    for all $t.$
\end{lemma}
\begin{proof}
    By item (3) of \cref{lemma:trans}, $\gam_I'(t)$ is non-zero
    and by \cref{lemma:allpos} it must be negative.
\end{proof}

\begin{lemma} \label{lemma:deltadrop}
    There is a constant $\delta > 0$ such that
    if $\gam$ is falling curve with $\gam_I(0) \in [h(c), c]$
    and $\length(\gam_M) \ge 1,$
    then
    \[
        \gam_I(0) - \gam_I(1) > \delta.
    \] \end{lemma}
\begin{proof}
    Consider
    the family of all falling curves $\gam : [0,1] \to M \ti (0, c]$
    where $\gam_I(0) \in [h(c), c]$ and where the
    projected curve $\gam_M : [0, 1] \to M$
    is parameterized by arc length
    (and therefore has length exactly one).
    All such curves are tangent to a continuous line field
    given by the transverse intersection in item (2) of \cref{lemma:trans}.
    Therefore in the $C^1$ topology,
    this family is closed and equicontinuous and
    so by the Arzel\`a-Ascoli theorem is compact.
    Since $\gam_I(0) - \gam_I(1)$
    depends continuously on $\gam,$
    it is bounded away from zero.
\end{proof}
\begin{lemma} \label{lemma:fundrop}
    There is a constant $L > 0$ such that
    if $\gam$ is falling curve with $\gam_I(0) \in [h(c), c]$
    and $\length(\gam_M) \ge L,$
    then
    $\gam_I(1) < h(c).$
\end{lemma}
\begin{proof}
    Take $L$ to be an integer large enough that $\delta L > c - h(c).$
    %greater than tfrac{c - h(c)}{delta}.
    Given a falling curve $\gam$ with projected length at least $L,$
    we can split $\gam$ into $L$ subcurves such that
    \cref{lemma:deltadrop} applies to each of them.
\end{proof}
Recall that the constant $0 < \eta < 1$ satisfies
$ \| Dg v^s \| \le \eta \| v^s \| $ for all $v^s \in \Esg.$

%FLOWS
%
%Sec:incoherent is unchanged except that phi_1 is used in place of g
%in the proof of lemma:falln.

\begin{lemma} \label{lemma:falln}
    For each $k,$ if $\bt$ is a falling curve with
    $\bt_I(0) \in [h^{k+1}(c), h^k(c)]$
    and $\length(\bt_M) > L \eta^k,$
    then $\bt_I(1) < h^{k+1}(c).$
\end{lemma}
\begin{proof}
    Consider the curve $\gam = f^{-k} \circ \bt.$
    It is a falling curve with $\gam_I(0) \in [h(c), c]$ and
    \[
        \length(\gam_M)
        \ = \
        \length(g^{-k} \bt_M)
        \ > \
        \eta^{-k} \length(\bt_M) \ > \ L.
    \]
    The previous lemma then shows that
    $\gam_I(1) < h(c)$ and thus
    $\bt_I(1) < h^{k+1}(c).$
\end{proof}
\begin{lemma} \label{lemma:fallshort}
    If $\gam$ is a falling curve, then
    \[
        \length(\gam_M) \le \frac{L}{1 - \eta}.
    \] \end{lemma}
\begin{proof}
    Split $\gam$ into a concatenation of falling curves
    $\gam_k$ indexed by some finite $S \subof \bbN$ and
    where each $\gam_k$ lies in $M \ti [h^{k+1}(c), h^k(c)].$
    Then
    \[
        \length(\gam_M)
        \ = \
        \sum_{k \in S} \length(\pi_M \gam_k)
        \ \le \
        \sum_{k \in \bbN} L \eta^{-k}
        \ = \
        \frac{L}{1 - \eta}.
        \qedhere
    \] \end{proof}
This shows that if we try to continue a falling curve as far as possible,
then after finite length the result is a curve that ends
at the attractor $M \ti \{0\}.$
We use this to show dynamical incoherence.

\begin{proof}
    [Proof of dynamical incoherence.]
    Suppose there is a foliation on $M \ti S^1$
    tangent to $\Ecuf.$
    Pick a point $(x,z) \in M \ti (0,c]$
    and define an immersed submanifold
    $S = \Wsg(x) \ti (-c, c).$
    By the transversality established in \cref{lemma:trans},
    the \emph{cu}-foliation on $M \ti S^1$ induces
    a one-dimensional foliation $\cF$ on $S.$
    Let $L$ be the leaf of $\cF$ through the point $(x, z)$ and
    let $\gam : [0, \infty) \to L$
    be a curve of infinite length starting at $\gam(0) = (x, z)$
    such that $\gam_M'(t)$ lies in $E^-$
    for all $t \ge 0.$
    \Cref{lemma:fallshort} implies that there is $T > 0$
    such that $\gam(T) \in \Wsg(x) \ti \{0\}.$

    \begin{figure}
        \centering
        \includegraphics{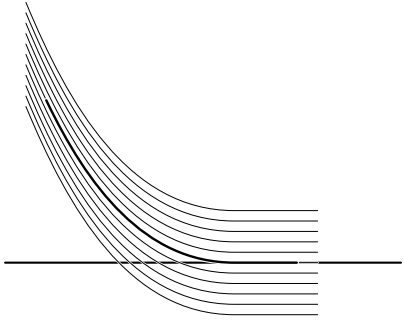}
        \caption{
        The foliation chart around $\gam([0,T])$
        considered in the proof of dynamical incoherence.
        The horizontal line represents $z = 0$ and the foliation 
        chart contains segments of leaves which cross this line.}
        \label{fig:folnbox}
    \end{figure}
    As $\gam([0,T])$ is compact,
    we can find a foliation chart of $\cF$ containing $\gam([0,T])$
    in its interior, as shown in \cref{fig:folnbox}.
    Consequently, there exists a curve
    $\al : [0, T] \to M \ti (-c, c)$
    lying in a single leaf of $\cF$ such that
    $\al$ starts above $M \ti \{0\}$
    and ends below $M \ti \{0\}.$
    That is,
    $\al_I(T) < 0 < \al_I(0).$
    Using both \cref{lemma:falling} and
    the symmetry of the definition of $f(x,-z)$ in terms of $f(x, z),$
    we can show that $\al_I'(t) > 0$
    for all $t$ with $\al_I(t) < 0$ and this gives a contradiction.
\end{proof}

\section{Anosov flows} \label{sec:flows} %{{{1

Say $\phi$ is an Anosov flow on a manifold $M.$
To keep notation consistent with partially hyperbolic diffeomorphisms,
we write the invariant splitting as
\[
    TM = \Esphi \oplus \Ecphi \oplus \Euphi
\]
where $\Esphi$ is the strong stable foliation of the Anosov flow,
$\Euphi$ is the strong unstable foliation, and
$\Ecphi$ is the one-dimensional subbundle given by the direction of the flow.
There are constants $0 < \lam < \eta < 1 < \mu$ such that
\[
    \lam < \| D \phi_1 v^s \| < \eta \qandq \mu < \| D \phi_1 v^u \|
\]
for any strong stable unit vector $v^s \in \Esphi$
and any strong unstable unit vector $v^u \in \Euphi.$
Take an integer $N > 0$ large enough that $\eta^N < \lam$
and define $g : M \to M$ to be the time-$N$ map
$g = \phi_N.$
Then $g$ is partially hyperbolic with exactly the same splitting of $TM$
as the Anosov flow
and $ \| Dg v^s \| < \lam$
for any stable unit vector $v^s \in \Esg = \Esphi.$
As in the construction in \cref{sec:construction},
define cone families $\Aone, \Bone, \Cone,$ and $\Uone$
based on the partially hyperbolic splitting of $g.$
We may freely assume that the metric on $M$ is adapted to the Anosov
flow and that the cone families are defined based on the Anosov splitting
so 
that the inclusions
\[
    \Aone \cc D \phi_t(\Aone),
    \quad
    D \phi_t(\Cone) \cc \Cone
    \qandq
    D \phi_t \Uone \cc \Uone
\]
hold for all $t > 0.$

As before,
define the vector field $X \subof \Aone^+$ and
the corresponding (non-Anosov) flow $\sig.$
Since we now have two flows, we will actually
use $X_\sig$ to denote the vector field generating $\sig$ and
use $X_\phi$ to denote the vector field generating $\phi.$
Up to rescaling $X_\sig,$
we may freely assume that
the composition $g_t = \sig_t \circ g$ satisfies
\[
    \| Dg_1 v^s \| < \lam < \| Dg_1 v^c \| < \mu < \| Dg_1 v^u \|
\]
for all unit vectors
$v^s \in \Es_{g_1},$ 
$v^c \in \Ec_{g_1},$ 
and
$v^u \in \Eu_{g_1}.$

\begin{figure}

\centering
\includegraphics{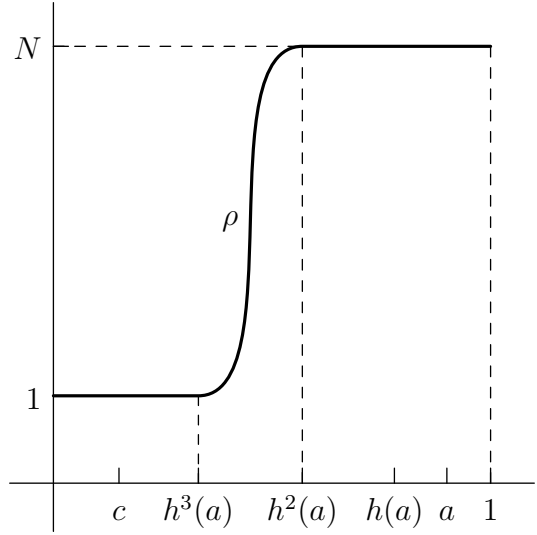}
\caption{The graph of the function $\rho$.}
\label{fig:rho}
\end{figure}
Define the functions $h$ and $\tau$ as before, with the caveat that
the constant $\lam$ used is now the one defined in the current section.
As depicted in \cref{fig:rho},
define a smooth bump function $\rho : [0, 1] \to [1, N]$ such that
\begin{enumerate}
    \item $\rho(z) = 1$ for $0 \le z \le h^3(a),$
    \item
    $\rho'(z) > 0$ for $h^3(a) < z < h^2(a),$ and
    \item
    $\rho(z) = N$ for $h^2(a) \le z \le 1.$
\end{enumerate}
Define $f : M \ti I \to M \ti I$ by
\[
    f(x, z) =
    \begin{cases}
        ( g_{\tau(z)}(x), h(z) )   & \text{if } z \ge h^2(a), \\
        ( \phi_{\rho(z)}(x), h(z) ) & \text{if } z \le h^2(a).
    \end{cases}  \]
This function is well defined at $z = h^2(a)$
and is as smooth as the Anosov flow $\phi.$
The function is partially hyperbolic on $M \ti \{0,1\}$
with the same splittings on $M \ti \{0\}$ and $M \ti \{1\}$
as stated in \cref{sec:construction}.
Importantly, it is
absolutely partially hyperbolic on the subset
$M \ti \{0,1\}$ since
\[
    \| Df v^s \| \le \lam < \| Df v^c \| < \mu < \| Df v^u \|
\]
for all points $p \in M \ti \{0,1\}$ and unit vectors
$v^s \in \Esf(p), v^c \in \Ecf(p), v^u \in \Euf(p).$

In adapting the proofs from \cref{sec:construction},
we have the additional complication that in the region
$M \ti [h^3(a), h^2(a)]$
vertical vectors are sheared in the direction of the Anosov flow, $\phi.$
However, since this shearing is exactly in the center direction
$\Ecg = \Ecphi,$ we can recover the proofs.

Consider a point $p = (x, z)$ in the fundamental domain
$M \ti (h(a), a],$ and a tangent vector
$(u, v) \in T_x M \ti \bbR$ is based at $p.$
In the region $M \ti [h^2(a), 1],$ the function $f$ is defined
exactly as before, so the formulas for $u_1$ and $u_2$
(as given in the proof of \cref{lemma:extendcu})
are unchanged.
After that,
\[
    u_3 \ =\ D \phi_r(u_2)\ +\ v_2 \cdot \rho'(z_2) \cdot X_\phi(x_3),
\]
where, for the rest of the section, $r$ denotes the value $r = \rho(z_2).$
This formula implies that
\[    
    u_3 \ \in\ D \phi_r(u_2)\ +\ \Ecg(x_3)
\]
where again we stress that $\Ecg = \Ecphi.$
For $n \ge 3,$
$u_{n+1} = D \phi_1(u_n).$

We now list the changes needed in order to adapt the proof
of \cref{thm:oneswitch} to give a proof of \cref{thm:anosovswitch}.

\begin{itemize}
    \item \textbf{Adapting the proof of \cref{lemma:extendcu}:}
    Combining $u_2 \in u_B + \Cone(x_2)$ with $u_3 \in D \phi_r(u_2) + \Ecg(x_3)$
    and using
    $\Ecg(x_3) = D \phi_r( \Ecg(x_2) )$
    yields
    \[
        u_3\ \in \ D \phi_r \big(u_B + \Cone(x_2) + \Ecg(x_2) \big)
            \ =\ D \phi_r \big(u_B + \Cone(x_2) \big).
    \]
    Since
    $D \phi_r : T_{x_2} M \to T_{x_3} M$ is a linear isomorphism
    and any element of $u_B + \Cone(x_2)$ is non-zero,
    $u_3$ is non-zero.
    Then $u_n = D \phi_1^{n-3}(u_3)$ for all $n \ge 3$ and
    the rest of the proof then follows as before
    with $\phi_1$ in place of $g.$
    \item
    \textbf{Adapting the proof of \cref{lemma:extendu}:}
    The fact that $u_n \in \Uone$ for all $n > 0$
    now relies on the property that
    $D \phi_t(\Uone) \subof \Uone$ for all $t > 0,$
    but otherwise the proof is unchanged.
    \item
    \textbf{Adapting the proof of \cref{lemma:trans}:}
    In the flow case,
    $Y_3 = D \phi_r(Y_2) \subof \Cone$
    and
    $Y_{n+1} = D \phi_1(Y_n)$
    for $n \ge 3.$
    Then as $D \phi_t (\Cone) \subof \Cone$ for any $t > 0,$
    it follows by induction that $Y_n \in \Cone$ for all $n \ge 0$
    and the rest of the proof is as before.
    \item
    \textbf{The proof of \cref{lemma:allpos}} is unchanged.
    \item
    \textbf{The proof of \cref{lemma:falln}}
    uses $\phi_1$ in place of $g,$ but otherwise \cref{sec:incoherent} is unchanged.
\end{itemize}

\section{Higher dimensional stable bundle} \label{sec:high} % {{{1

We now adapt the proofs to the case of a higher-dimension stable bundle.
We first look at the case where the stable bundle has a splitting into 
a ``strong stable'' bundle $\Essg$ of any dimension and
a ``weak stable'' bundle $\Ewsg$ which must be one-dimensional.
\begin{prop} \label{prop:highswitch}
    Suppose $g : M \to M$ is a partially hyperbolic diffeomorphism such that
    the stable bundle has a subsplitting
    \[
        \Esg = \Essg \oplus_< \Ewsg
    \]
    where dim $\Ewsg = 1$ and $g$ preserves an orientation of $\Ewsg.$
    Then there is a partially hyperbolic diffeomorphism
    $f : M \ti S^1 \to M \ti S^1$ such that
    \begin{enumerate}
        \item on the invariant submanifold $M \ti \{0\},$
        the partially hyperbolic splitting is given by
        \[
            \Esf = \Essg \ti \bbR,
            \quad
            \Ecf = (\Ewsg \oplus \Ecg) \ti 0,
            \qandq
            \Euf = \Eug \ti 0;
        \]
        \item
        $f(x,0) = (g(x), 0)$ for all $x \in M;$ and
        \item
        $f$ is dynamically incoherent.
    \end{enumerate} \end{prop}
In this section, we adapt the earlier arguments in order to prove the
proposition. Then in the next section, we use the proposition to prove
\cref{thm:multiswitch}.
Therefore, for the remainder of this section, assume $g$
satisfies the hypotheses of \cref{prop:highswitch}.

The biggest modification is the introduction of a cone family $\Eone$
which is transverse to $\Essg.$
In the proofs, we ensure that all of the vectors of interest stay inside $\Eone$
and are therefore bounded in angle away from $\Essg.$
In the original proof of \cref{thm:oneswitch} given in \cref{sec:construction},
we first chose a cone family $\Aone$ associated to $\Ecug$ and then chose a smooth
vector field $X$ lying inside of $\Aone.$
As we will explain later, in order to use \cref{prop:highswitch} as a step in proving
\cref{thm:multiswitch}, we need to reverse these steps and choose $\Aone$ and $\Eone$
based on a given $X.$

Let $X$ be a smooth vector field transverse to
$\Essg \oplus \Ecug;$
that is, $X(x) \notin \Essg \oplus \Ecug$ for all $x \in M.$

\begin{lemma} \label{lemma:xgiven}
    There are cone fields
    $\Aone$ associated to $\Esg$ and
    $\Eone$ associated to $\Ewsg \oplus \Ecug$
    such that
    \begin{enumerate}
        \item $X$ lies in $\Aone \cup Dg^2(\Eone)$ and
        \item
        the intersection $Dg^2(\Aone) \cap \Eone \cap (\Essg \oplus \Ecug)$
        consists only of zero vectors.
    \end{enumerate} \end{lemma}
\begin{remark}
    Here, the associations mean that
    \[
        \Esg \cc \Aone \cc Dg(\Aone), \quad \Ecug \cc \Aone^*, \quad
        \Ewsg \oplus \Ecug \cc Dg(\Eone) \cc \Eone,
        \quad
        \Essg \cc \Eone^*.
    \] \end{remark}
\begin{proof}
    For this proof, we can decompose a tangent vector $v \in TM$
    with respect to the splitting
    $TM = \Essg \oplus \Ewsg \oplus \Ecug$
    and write the components as
    $v = \vsss + \vws + \vcu.$
    By assumption, if $v = X(x)$ for some $x \in M,$
    then $\vws$ is non-zero.
    Therefore, there is a constant $C > 1$ such that
    $ \| \vsss \| ^2 + \| \vcu \| ^2 < C \| \vws \| ^2$
    for every such vector $v = X(x)$ in the vector field.
    Up to replacing $C$ by a larger constant, we can also assume that this
    inequality holds for any vector of the form
    $v = Dg^{-2}(X(x))$ where $x \in M.$
    With $C$ determined, 
    define the cone fields $\Aone$ and $\Eone$ by
    \[
        v \in Dg^2(\Aone)
        \quad \Leftrightarrow \quad
        \| \vcu \| ^2 \le C \| \vws \| ^2 + 2 \| \vsss \| ^2
    \]
    and
    \[
        v \in \Eone
        \quad \Leftrightarrow \quad
        \| \vsss \| ^2 \le C \| \vws \| ^2 + 2 \| \vcu \| ^2.
    \]
    We assume that the metric on $M$ is adapted to the splitting
    so that the inclusions $\Aone \cc Dg(\Aone)$ and $Dg(\Eone) \cc \Eone$
    both hold.
    It then follows that $X \subof Dg^2(\Aone) \cc \Aone$ and
    $Dg^{-2}(X) \subof \Eone.$
    Together, these imply that $X \subof \Aone \cap Dg^2(\Eone).$

    A vector $v \in Dg^2(\Aone) \cap \Eone \cap (\Essg \oplus \Ecug)$
    has $ \| \vws \| = 0$ and therefore
    \[
        \| \vcu \| ^2 \le 2 \| \vsss \| ^2
        \qandq
        \| \vsss \| ^2 \le 2 \| \vcu \| ^2,
    \]
    showing that $v$ is a zero vector.
\end{proof}
Let $\Aone$ and $\Eone$ be as in the lemma.
As before, define $\Bone = Dg^2(\Aone).$
In the original proof in \cref{sec:construction}, we defined $\Cone$ as $Dg(\Bone^*).$
Here, we need to be slightly more careful.

\begin{lemma} \label{lemma:avoidss}
    There is a cone family $\Cone$ associated to $\Ecug$ such that
    if $0 \ne u \in \Bone \cap \Eone$
    and
    $v \in \Cone,$
    then
    $u + v \notin \Essg.$
\end{lemma}
\begin{remark}
    Here,
    we assume $u$ and $v$ are both in the same tangent space $T_x M$
    in order for the sum $u + v$ to be defined.
\end{remark}
\begin{proof}
    We claim that $\Cone$ can be defined as $Dg^n(\Bone^*)$ for some $n \ge 1.$
    If not, then
    there are sequences $\{u_n\}$ and $\{v_n\}$ of non-zero vectors
    such that for all $n \ge 1$
    \[
        u_n \in \Bone \cap \Eone, \quad
        v_n \in Dg^n(\Bone^*), \qandq
        u_n + v_n \in \Essg.
    \]    %
    Up to rescaling the vectors, assume their norms satisfy
    max $\{ \| u_n \| , \| v_n \| \} = 1.$
    By taking subsequences, we can reduce to the case
    where both sequences converge.
    Write
    $u = \lim_{n \to \infty} u_n \in \Bone \cap \Eone,$
    and
    $v = \lim_{n \to \infty} v_n \in \Ecug.$
    Then as $u + v \in \Essg,$ it follows that
    $u = (u + v) - v$ is non-zero and lies in both
    $\Bone \cap \Eone$ and $\Essg \oplus \Ecug.$
    This contradicts the previous lemma.
\end{proof}

We henceforth assume $\Cone$ is as in the lemma.
Let $E^+$ and $E^-$ denote the two connected components of
$\Ewsg \sans 0.$
Then $E^+ + \Essg$ and $E^- + \Essg$ are the two connected components of
$\Esg \sans \Essg.$
The intersection $\Aone \cap Dg(\Eone)$
is a neighbourhood of the one-dimensional bundle $\Ewsg$
and is transverse to the codimension one bundle
$\Essg \oplus \Ecug.$
Therefore $(\Aone \cap Dg(\Eone)) \sans 0$
has two connected components which we denote by $\Aone^+$ and $\Aone^-$
Here, $\Aone^+$ is the component containing $E^+.$
Similarly, $(\Bone \cap \Eone) \sans 0$
has two connected components $\Bone^+$ and $\Bone^-.$
Under these new definitions, $\Aone^+$ is not necessarily a subset
of $Dg(\Aone^+).$ However, the following properties still hold:
\begin{enumerate}
    \item $\Aone^+ \cc \Bone^+$ and $Dg(\Aone^+) \cc \Bone^+,$
    \item
    each of $\Aone^+$ and $\Bone^+$ is closed under addition
    and under scaling by positive real numbers,
    \item
    the intersection $\Bone^+ \cap \Cone$ is empty.
\end{enumerate}
These are the only properties of $\Aone^+$ and $\Bone^+$
that were used in the earlier proofs.
However, to adapt the proof of \cref{lemma:extendcu} to this new setting, we will
also need the following corollary of \cref{lemma:avoidss}.

\begin{cor} \label{cor:bcss}
    If $u \in \Bone^+$ and $v \in \Cone,$ then $u + v \notin \Essg.$
\end{cor}
Using the notation for adding subsets of $TM$ introduced in \cref{sec:construction},
we can rewrite this statement as
$(\Bone^+ + \Cone) \cap \Essg = \varnothing$
from which it follows that
any vector in $\Bone^+ + \Cone + \Essg$ is non-zero.

As before, use the vector field $X \subof \Aone^+$
to define a flow $\sig_t.$
After rescaling $X,$
that is, replacing $X$ by $\ep X$ for a sufficiently small $\ep,$
the composition $g_t = \sig_t \circ g$ satisfies
$Dg_t \Eone \subof \Eone$ for all $|t| \le 1$
as well as all of the inclusions for $g_t$ given in \cref{sec:construction}.
Consequently,
$Dg_t(\Aone^+) \cc \Bone^+$
for all $|t| \le 1.$

We can now adapt the proofs in \cref{sec:construction} to this setting.

\begin{itemize}
    \item \textbf{Adapting the proof of \cref{lemma:extendcu}:}
    The original proof (using the new cone families)
    shows that $u_2$ is a sum of a vector in $\Bone^+$
    with a vector in $\Cone,$ and therefore $u_2 \notin \Essg$
    by \cref{cor:bcss}
    Consequently as $n_j \to +\infty,$
    $\{w^{n_j}\}$ converges to a vector
    outside of $\Essg \ti \bbR.$

    \noindent
    This adaptation of the proof establishes an additional 
    property that we will need to use later:
    \begin{quote}
        for all $(x, z) \in M \ti (0, h^2(a)],$ \
        $\Ecuf(x, z)$ is transverse to $\Essg(x) \ti \bbR.$ \end{quote}
    \item
    \textbf{The proofs of lemmas \ref{lemma:extendu} and \ref{lemma:trans}}
    are unchanged.
    \item
    \textbf{The changes to \cref{lemma:allpos}}
    are significant enough that for clarity
    we state and prove a new version of the lemma below.
\end{itemize}
\begin{lemma} \label{lemma:higherallpos}
    Consider a point 
    $(x, z) \in M \ti (0, h^2(a)]$ and a vector
    $(u,v) \in \Ecuf(x, z)$ such that $u \in \Esg(x).$
    Then
    $u \in E^+ + \Essg$
    if and only if $v > 0.$
\end{lemma}
\begin{proof}
    Define a function
    $\hat u : M \ti (0, h^2(a)] \to \Esg \sans \Essg$ 
    as follows:
    for each point $(x, z) \in M \ti (0, h^2(a)],$
    let $\hat u(x, z)$ be the unique $u \in \Esg(x)$
    such that $(u, 1) \in \Ecuf(x, z).$
    \Cref{lemma:trans} shows that such a vector exists and is unique.
    As we noted when adapting \cref{lemma:extendcu},
    $\Ecuf(x, z)$ is transverse to $\Essg(x) \ti \bbR,$
    and therefore $\hat u(x, z) \notin \Essg.$
    By continuity, 
    $\hat u$ 
    only takes values in a single connected component
    of $\Esg \sans \Essg,$ 
    and so to prove the lemma
    it suffices to show that $\hat u(x, z) \in E^+ + \Essg$ for a single point in
    $M \ti (0, h^2(a)].$

    Consider a point $(x_2, z_2)$ in the subset $M \ti (h^3(a), h^2(a))$
    and define
    \[    
        u_2 = \hat u(x_2, z_2) \in \Esg(x_2) \qandq v_2 = 1
    \]        
    so that $(u_2, v_2) \in \Ecuf(x_2, z_2).$
    From these, define
    \[
        (x, z) = f^{-2}(x_2, z_2)
        \qandq
        (u, v) =  Df^{-2}(u_2, v_2).
    \]
    It follows that 
    $(x, z) \in M \ti (h(a), a), (u, v) \in \Ecuf(x, z),$ and $v > 0.$
    Recall from the remark after
    the proof of \cref{lemma:extendcu}
    that
    $u_2 = Dg_{t_1}(u_C) + u_B$
    with $Dg_{t_1}(u_C) \in \Cone$ and $u_B \in \Bone^+.$
    If $u_2 \in E^- + \Essg,$ then
    \[
        0 \ = \ u_B - u_2 + Dg_{t_1}(u_C)
        \ \in \ \Bone^+ + (E^+ + \Essg) + \Cone.
    \]
    This is a contradiction, since
    \[
        \Bone^+ + (E^+ + \Essg) + \Cone
        \ = \ (\Bone^+ + E^+) + \Essg + \Cone
        \ = \ \Bone^+ + \Essg + \Cone
    \]
    does not have a zero vector.
    Therefore, $u_2 \in E^+ + \Essg.$
\end{proof}

With this established, we can recover the proof of dynamical incoherence
in \cref{sec:incoherent} with only minor changes.
In fact, the only changes are:
\begin{itemize}
    \item \textbf{The proof of \cref{lemma:falling}} now uses \cref{lemma:higherallpos}
    in place of \cref{lemma:allpos}.
    \item
    \textbf{In the two places where $E^-$ is used}
    (first in the definition of falling curve
    and then in the proof of dynamical incoherence),
    the set $E^-$ is replaced with $E^- + \Essg.$
\end{itemize}
This completes the proof of \cref{prop:highswitch}.

\medskip{}

In order to prove \cref{thm:multiswitch}, we establish two additional properties
of the construction in \cref{prop:highswitch}. The first concerns orientability.

\begin{addendum} \label{addendum:orient}
    In \cref{prop:highswitch}, if $\Essg$ has an orientation preserved by
    $g,$ then $f$ can be constructed so that it preserves the orientation
    of $\Esf.$
\end{addendum}
\begin{proof}
    Consider $f$ as a diffeomorphism of $M \ti [-1, 1]$ before we
    glue together the boundary components to produce a diffeomorphism on
    $M \ti S^1.$
    It is clear on the subset $M \ti \{0\}$ that $\Esf = \Essg \ti \bbR$ is oriented,
    and this extends to an orientation on all of $M \ti [-1, 1].$
    The problem is that the gluing which identifies
    $M \ti \{-1\}$ with $M \ti \{1\}$
    might reverse the orientation of $\Esf$ and therefore
    produce a stable bundle which is not orientable.
    (In fact, in the original construction in \cite{RHRHU-nondyn}
    it is easy to see that this is the case.)

    To avoid this problem, we first extend $f$ to a diffeomorphism of
    $M \ti [-1, 3]$ by requiring that if
    $f(x, z) = (x_1, z_1)$ for $(x, z) \in M \ti [-1,1],$ then
    $f(x, z+2) = (x_1, z_1 + 2).$
    Then identify $M \ti \{-1\}$ with $M \ti \{3\}$ to produce a map on $M \ti S^1.$
    Under this construction, $\Esf$ is oriented.
\end{proof}
\begin{addendum} \label{addendum:onesplit}
    In \cref{prop:highswitch}, if $\Esg$
    has a splitting into one-dimensional subbundles
    each with orientation preserved by $g,$
    then $f$ can be constructed so that its stable bundle $\Esf$
    has the same property.
\end{addendum}
\begin{proof}
    Write the dominated splitting of $g$ as
    \[
        TM = \Es_d \oplus_< \cdots \oplus_< \Es_2 \oplus_< \Es_1
        \oplus_< \Eccg \oplus_< \Eug.
    \]
    Here, we use $\Eccg$ to denote the center bundle of $g$
    and call it the ``true center'' as we will group it together 
    with some of the stable bundles to produce higher-dimensional centers.
    Specifically, for each $1 \le k \le d,$
    define a 4-way splitting of $TM$ by
    \[
        \Essg = \Es_d \oplus \cdots \oplus \Es_{k+1},
        \qquad
        \Ewsg = \Es_k,
        \qquad
        \Ecg = \Es_{k+1} \oplus \cdots \oplus \Es_1 \oplus \Eccg,
        \qquad
        \Eug = \Eug.
    \]
    Our goal is now to show that the proof of \cref{prop:highswitch}
    applies for each value of $k$ and with the exact same
    definition of $f$ in each case.
    This will imply for all $1 \le k \le d$
    that $f$ has a 3-way partially hyperbolic splitting
    with a stable bundle of dimension $d - k + 1.$
    This is equivalent to $f$ having a $d$-dimensional stable bundle
    that splits into one-dimensional subbundles.
    The orientability of these subbundles follows from
    the previous addendum.

    Fix constants $0 < \lam < \eta < 1 < \mu$ independent of $k$ such that
    \[
        \lam < \| Dg v^s \| < \eta \qandq \mu < \| Dg v^u \|
    \]
    for all unit vectors $v^s \in \Es_d \oplus \cdots \oplus \Es_1$ and $v^u \in \Eug.$
    Then as in \cref{sec:construction},
    define functions $h$ and $\tau$ based on these constants.

    We also need to make a single choice of vector field $X.$
    For each $k,$ let $X_k$ be the (not necessarily smooth) unit vector field
    pointing in the positive direction of $\Es_k.$
    The sum $X_1 + \cdots + X_d$ is a vector field which is transverse to
    \[
        \Es_d \oplus \cdots \oplus E_{k+1} \oplus E_{k-1} \oplus \cdots \oplus \Es_1
        \oplus \Eccg \oplus \Eug
    \]
    for every $k.$
    We can approximate $X_1 + \cdots + X_d$ by a smooth vector field
    $X$ which has the same transversality property.
    The proof of \cref{prop:highswitch}
    replaces $X$ with a rescaling $\ep X$ for some small positive $\ep,$
    so that the resulting flow $\sig_t$ has nice properties
    for $|t| \le 1.$
    Here, we can choose $\ep$ small enough so that same rescaled
    vector field works in the proof of \cref{prop:highswitch}
    for all $1 \le k < d.$

    With $h$, $\tau$, and $X$ specified
    and with $X$ determining $\sig_t$ and therefore
    $g_t = \sig_t \circ g,$ we see that the exact same diffeomorphism $f$
    given by the formula
    $f(x, z) = ( g_{\tau(z)}(x), h(z) )$
    works for all $1 \le k \le d.$
\end{proof}

\section{Applying multiple switches} \label{sec:multi} %{{{1

With \cref{prop:highswitch} established, we use it to prove \cref{thm:multiswitch}.

\begin{proof}
    [Proof of \cref{thm:multiswitch}]
    Let $g_0 : M_0 \to M_0$ be a partially hyperbolic diffeomorphism
    where the stable bundle
    has an invariant dominated splitting into one-dimen\-sio\-nal subbundles.
    We write this as
    \[
        \Es_{g_0} = \Es_d \oplus_< \cdots \oplus_< \Es_2 \oplus_< \Es_1.
    \]
    %where Es_d has the strongest contraction and
    %Es_1 has the weakest contraction.
    Assume $g_0$ preserves an orientation of each $\Es_k.$
    Repeatedly applying \cref{prop:highswitch} with addendum \ref{addendum:onesplit},
    we construct a sequence of partially hyperbolic maps
    $g_k : M_k \to M_k$
    where $M_k = M_{k-1} \ti S^1 = M_0 \ti \bbT^k$ for $0 < k \le d.$

    On the submanifold $M_0 \ti \{0\} \subof M_0 \ti \bbT^k,$
    \[
        \Es_{g_k} = \big(\Es_d \oplus \cdots \oplus \Es_{k+1} \big) \ti T_0 \bbT^k,
        \quad
        \Ec_{g_k} = \big(\Es_k \oplus \cdots \oplus \Es_1 \oplus \Ec_{g_0} \big) \ti 0,
        \qandq
        \Eu_{g_k} = \Eu_{g_0} \ti 0.
    \]
    Hence, if we take $f = g_d,$ then
    \[
        \Esf = 0 \ti T_0 \bbT^d,
        \quad
        \Ecf = \big( \Es_{g_0} \oplus \Ec_{g_0} \big) \ti 0
        \qandq
        \Euf = \Eu_{g_0}.
    \]
    The last application of \cref{prop:highswitch},
    going from $g_{d-1}$ to $g_d,$
    implies that $f = g_d$ is dynamically incoherent
    and so the conclusions of \cref{thm:multiswitch} hold.
\end{proof}
% Acknowledgements

\medskip

\acknowledgement
This research was partially funded by the Australian Research Council.
The author thanks Rafael Potrie, Davide Ravotti,
Federico Rodri\-guez-Hertz, Jana Rodriguez-Hertz, and Ra\'ul Ures 
for helpful discussions.

% Epilogue {{{1

\bibliographystyle{alpha}
\bibliography{dynamics}

\def\cprime{$'$}
\begin{thebibliography}{CRHRHU18}

\bibitem[BBI09]{BBI2}
M.~Brin, D.~Burago, and S.~Ivanov.
\newblock Dynamical coherence of partially hyperbolic diffeomorphisms of the
  3-torus.
\newblock {\em Journal of Modern Dynamics}, 3(1):1--11, 2009.

\bibitem[BDV05]{buhbook}
Christian Bonatti, Lorenzo~J. D\'{\i}az, and Marcelo Viana.
\newblock {\em Dynamics beyond uniform hyperbolicity}, volume 102 of {\em
  Encyclopaedia of Mathematical Sciences}.
\newblock Springer-Verlag, Berlin, 2005.
\newblock A global geometric and probabilistic perspective, Mathematical
  Physics, III.

\bibitem[BFFP19]{1908.06227}
Thomas Barthelm\'{e}, Sergio~R. Fenley, Steven Frankel, and Rafael Potrie.
\newblock Partially hyperbolic diffeomorphisms homotopic to the identity in
  dimension 3, {P}art {I}: {T}he dynamically coherent case, 2019.
\newblock https://arxiv.org/abs/1908.06227.

\bibitem[BFFP20]{2008.04871}
Thomas Barthelm\'{e}, Sergio~R. Fenley, Steven Frankel, and Rafael Potrie.
\newblock Partially hyperbolic diffeomorphisms homotopic to the identity in
  dimension 3, {P}art {II}: {B}ranching foliations, 2020.
\newblock https://arxiv.org/abs/2008.04871.

\bibitem[BGHP20]{BGHP}
Christian Bonatti, Andrey Gogolev, Andy Hammerlindl, and Rafael Potrie.
\newblock Anomalous partially hyperbolic diffeomorphisms {III}: {A}bundance and
  incoherence.
\newblock {\em Geom. Topol.}, 24(4):1751--1790, 2020.

\bibitem[CP15]{cropot2015lecture}
S.~Crovisier and R.~Potrie.
\newblock Introduction to partially hyperbolic dynamics.
\newblock Unpublished course notes available online, 2015.

\bibitem[CRHRHU18]{crhrhu2018survey}
Pablo~D. Carrasco, Federico Rodriguez-Hertz, Jana Rodriguez-Hertz, and Ra\'{u}l
  Ures.
\newblock Partially hyperbolic dynamics in dimension three.
\newblock {\em Ergodic Theory Dynam. Systems}, 38(8):2801--2837, 2018.

\bibitem[Ham18]{ham2018construct}
Andy Hammerlindl.
\newblock Constructing center-stable tori.
\newblock {\em Ann. Inst. H. Poincar\'{e} Anal. Non Lin\'{e}aire},
  35(3):713--728, 2018.

\bibitem[HHU11]{RHRHU-tori}
F.R. Hertz, M.A. Hertz, and R.~Ures.
\newblock Tori with hyperbolic dynamics in 3-manifolds.
\newblock {\em Journal of Modern Dynamics}, 5(1):185--202, 2011.

\bibitem[HP15]{HP2}
A.~Hammerlindl and Rafael Potrie.
\newblock Classification of partially hyperbolic diffeomorphisms in 3-manifolds
  with solvable fundamental group.
\newblock {\em J. Topol.}, 8(3):842--870, 2015.

\bibitem[HP18]{hp2018survey}
Andy Hammerlindl and Rafael Potrie.
\newblock Partial hyperbolicity and classification: a survey.
\newblock {\em Ergodic Theory Dynam. Systems}, 38(2):401--443, 2018.

\bibitem[HP19]{hp2019class}
Andy Hammerlindl and Rafael Potrie.
\newblock Classification of systems with center-stable tori.
\newblock {\em Michigan Math. J.}, 68(1):147--166, 2019.

\bibitem[HW99]{hw1999prevalence}
Boris Hasselblatt and Amie Wilkinson.
\newblock Prevalence of non-{L}ipschitz {A}nosov foliations.
\newblock {\em Ergodic Theory Dynam. Systems}, 19(3):643--656, 1999.

\bibitem[New70]{newhouse1970codimension}
S.~E. Newhouse.
\newblock On codimension one {A}nosov diffeomorphisms.
\newblock {\em Amer. J. Math.}, 92:761--770, 1970.

\bibitem[RHRHU16]{RHRHU-nondyn}
F.~Rodriguez~Hertz, M.~A. Rodriguez~Hertz, and R.~Ures.
\newblock A non-dynamically coherent example on {$\Bbb{T}^3$}.
\newblock {\em Ann. Inst. H. Poincar\'{e} Anal. Non Lin\'{e}aire},
  33(4):1023--1032, 2016.

\end{thebibliography}

\end{document}